\documentclass[12pt,a4paper]{article}
\usepackage[utf8]{inputenc}
\usepackage[english]{babel}
\usepackage{amsmath}
\usepackage{amsfonts}
\usepackage{amssymb}
\usepackage[left=1.5cm,right=1.5cm,top=2cm,bottom=2cm]{geometry}
\usepackage{amsthm}
\usepackage{multicol} 
\usepackage{latexsym}
\newcommand{\Div}[1]{\text{div}\,}

\newtheorem{cond}{{Condition}}
\newtheorem{remark}{{Remark}}
\newtheorem{lemma}{{Lemma}}[section]
\newtheorem{corollary}{Corollary}
\newtheorem{theorem}{{Theorem}}[section]

\numberwithin{equation}{section}
\title{On the Existence and Uniqueness of Solution of Boundary-Domain Integral Equations for the Dirichlet Problem for the Non-Homogeneous Heat Transfer Equation defined on a 2D Unbounded Domain}
\author{Z.W. Woldemicheal\footnote{First author}, C.Fresneda-Portillo}

\begin{document}\maketitle
\begin{center}
The authors gratefully acknowledge the financial support received from the London Mathematical Society Scheme 5, Award 51809, for Collaboration with Developing Countries.
\end{center}

\abstract{A system of segregated boundary-domain integral equations (BDIEs) is obtained from the Dirichlet problem for the diffusion equation in non-homogeneous media defined on an exterior two-dimensional domain. We use a parametrix different from the one employed by in \cite{dufera}. The system of BDIEs is formulated in terms of parametrix-based surface and volume potentials whose mapping properties are analysed in weighted Sobolev spaces. The system of BDIEs is shown to be equivalent to the original boundary value problem and uniquely solvable in appropriate weighted Sobolev spaces suitable for infinite domains. }

\section{Introduction}
Boundary Domain Integral Equations appear naturally when applying the Boundary Integral Method to boundary value problems with variable coefficient. This class of boundary value problems (BVPs) has a wide range of applications in Science  and Engineering, such as, heat transfer in non-homogeneous media \cite{ravnik}, motion of laminar fluids with variable viscosity \cite{carlosstokes}, or even in acoustic scattering\cite{acoustics}. 

The popularity of the Boundary Integral Method\cite{steinbach} is due to the reduction of the discretisation domain for boundary value problems with a homogeneous PDE and constant coefficients. For example, if the boundary value problem (BVP) is defined on a three dimensional domain, then, the boundary integral method reduces the BVP to an equivalent system of boundary integral equations (BIEs) defined only on the \textit{boundary} of the domain. However, this requires an explicit fundamental solution of the partial differential equation appearing in the BVP. Although these fundamental solutions may exist, they might not always be available explicitly for PDEs with variable coefficients. To overcome this obstacle, one can construct a \textit{parametrix} using the known fundamental solution. A discussion on fundamental solution existence theorems, algorithms for constructing fundamental solutions and parametrices is available in \cite{pomp}.  Classical examples of derivation of  Boundary Domain Integral Equations are:  for the diffusion equation\cite{mikhailov1} with variable coefficient in bounded domains in $\mathbb{R}^{3}$; same problem applying a different parametrix\cite{carloscomp}; the Dirichlet problem\cite{carloszenebe} in $\mathbb{R}^{2}$ and the mixed problem for the compressible Stokes system\cite{carlosstokes}, which is a good example of the application of the BDIE method to a PDE system. 

\textit{In this paper,} we explore a new family of parametrices for the operator 
\begin{equation}
\mathcal{A}u(x):= \sum_{i=1}^{2}\dfrac{\partial}{\partial x_{i}}\left(a(x)\dfrac{\partial u(x)}{\partial x_{i}}\right).\label{op1A}
\end{equation} 
of the form $$P^{x}(x,y)= P(x,y;a(x))=\dfrac{1}{2\pi a(x)}\mathrm{log}|x-y|$$
which can be useful at the time of studying BDIES derived from a BVP with a system of PDEs with variable coefficient as illustrated in \cite[Section 1]{carloscomp}. Note that this parametrix is different from the parametrix $P^{y}(x,y)$\cite{dufera}
\begin{equation}
P^{y}(x,y)= P(x,y;a(y)):=\dfrac{1}{2\pi a(y)}\mathrm{log}|x-y|,\,\, x, y \in \mathbb{R}^{2}.\label{pary2d}
\end{equation}

In particular, the work presented in this paper, will provide a method to obtain an equivalent system of BDIEs even when the single layer potential is not invertible and the domain of the BVP is unbounded. Although, there is some preliminary work related to BDIEs in two dimensional domains, see \cite{dufera}, this only relates to the family of parametrices $P^{y}(x,y)$ and therefore, the corresponding analysis for the family $P^{x}(x,y)$ in two dimensions is a problem that remains open, and that is the main purpose of this paper. This study aims to continue the work in \cite{carloszenebe} and will motivate the study of BDIEs for the Stokes system in 2D.

The theoretical study of parametrices which include the variable coefficient depending on different variables is helpful at the time of deriving BDIES for boundary value problems for systems of PDEs. For example, the parametrix for the Stokes system in three dimensions involves the variable viscosity coefficient with respect to $x$ and also with respect to $y$, see \cite{carlosstokes}. 

The numerical implementation of algorithms to solve BDIES in two dimensions\cite{2dnumerics, numerics2d} has shown that it is possible to obtain linear convergence with respect to the number of quadrature curves, and in some cases, exponential convergence. Moreover, there is analogous research in 3D which shows the successful implementation of fast algorithms to obtain the solution of boundary domain integral equations, see \cite{ravnik, numerics, sladek}. Therefore, we believe this method brings new techniques to solve inverse boundary value problems with variable coefficients that can be computationally implemented in an efficient fashion. Despite the success of the numerical implementations, some authors\cite{2dnumerics} highlight that there is not much research in the literature related to the theory or numerical solution of boundary-domain integral equations in 2D. 

In order to study the possible numerical advantages of the new family of parametrices of the form $P^{x}(x,y;a(x))$ with respect to the parametrices already studied, it is necessary to prove the unique-solvability of an analogous BDIES derived with this new family of parametrices which has not yet been done for the bidimensional exterior Dirichlet problem for the diffusion equation with variable coefficient.  

In unbounded domains, the Dirichlet problem is set in weighted Sobolev spaces to allow constant functions in unbounded domains to be possible solutions of the problem. Hence, all the mapping properties of the parametrix based potential operators are shown in weighted Sobolev spaces. 

An analysis of the uniqueness of the BDIES is performed by studying the Fredholm properties of the matrix operator which defines the system. Unlike for the case of bounded domains, the Rellich compactness embeding theorem is not available for Sobolev spaces defined over unbounded domains. Nevertheless, we present a lemma to reduce the remainder operator to two operators: one invertible and one compact. Therefore, we can still benefit from the Fredholm Alternative theory to prove uniqueness of the solution.

\section{Basic Notations and Spaces}
Let $\Omega=\Omega^{+}$ be an unbounded domain in $\mathbb{R}^{2}$ and let $\Omega^{-}:=\mathbb{R}^{2}\smallsetminus\overline{\Omega^{+}}$ be the complementary set of $\Omega^{+}$. Note that $\Omega^{-}$ is a bounded and open domain. Let us denote the boundary of $\Omega^{+}$ by $S$. We assume that $\partial \Omega$ is simply connected, compact and of class $\mathcal{C}^{1}(\mathbb{R}^{2})$. 

Let us define the partial differential operator related with the diffusion equation in non-homogeneous media in two dimensions. 
\begin{equation}\label{ch4operatorA}
\mathcal{A}u(x):=\sum_{i=1}^{2}\dfrac{\partial}{\partial x_{i}}\left(a(x)\dfrac{\partial u(x)}{\partial x_{i}}\right)\,x\in \Omega,
\end{equation}
 where $u(x)$ is the unknown function. The coefficient $a(x)$ is a given function. It is easy to see that if $a\equiv 1$ then, the operator $\mathcal{A}$ becomes $\Delta$, the Laplace operator. 
 
 \underline{Throughout the paper}, we will assume $a(x)\in \mathcal{C}^{1}(\overline{\Omega})\cap L^{\infty}(\Omega)$ and that there exist two constants, $C_{1},C_{2}\in \mathbb{R}$,  such that:
\begin{equation*}\label{cond1}
0 < C_{1} < a(x) < C_{2}.
\end{equation*}

  We will use the following function spaces in this paper (see e.g. \cite{mclean, lions, hsiao} for more details). Let $\mathcal{D}'(\Omega)$ be the Schwartz distribution space; $H^{s}(\Omega)$ and $H^{s}(S)$ with $s\in \mathbb{R}$, the Bessel potential spaces; the space $H^{s}_{K}(\mathbb{R}^{2})$ consisting of all the distributions of $H^{s}(\mathbb{R}^{2})$ whose support is inside of a compact set $K\subset \mathbb{R}^{2}$; the spaces consisting of distributions in $H^{s}(K)$ for every compact $K\subset \overline{\Omega^{-}},\hspace{0.1em}s\in\mathbb{R}$. Last, let $\widetilde{H}^{s}(\Omega)=\lbrace g\in H^{s}(\mathbb{R}^{2}): {\rm supp}(g)\subset \overline{\Omega}\rbrace$.

Sobolev spaces and Bessel potential spaces defined on an unbounded domain become quite restrictive as the class of constant functions no longer belongs to these spaces. This is easy to see as the function $h(x)=1$ is not $L^{2}(\Omega)$ integrable since the area (measure) of $\Omega$ is not finite. However, by introducing weighted Sobolev spaces, we embed the class of constant functions into the Sobolev space\cite[Lemma 1]{dufera}. 

We will now introduce the following weighted Sobolev spaces\cite{exterior, KMWext20, carlosext} which are useful when dealing with exterior problems, since constant  functions are allowed to be solutions of the problem. In order to define weighted Sobolev spaces in $\mathbb{R}^{2}$, we will make use of the weight\cite{dufera} $\omega_{2}:\Omega\longrightarrow\mathbb{R}_{+}$ given by $\omega_{2}(x) = (1+\vert x\vert^{2})^{\frac{1}{2}}\mathrm{ln}(2+\vert x \vert^{2}).$ Using $\omega_{2}$, we can define the following weighted spaces\cite{giroire2,hanouzet,exterior} 
\begin{itemize}
   \item Weighted Lebesgue space $L_{2}(\omega_{2};\Omega):=\lbrace f \,\, \vert \,\,\, \omega_{2}f\in L_{2}(\Omega)\rbrace$;
   \item Sobolev weighted space \begin{equation} \label{wsob}
 \mathcal{H}^{1}(\Omega):=\lbrace f\in L_{2}(\omega^{-1}_{2};\Omega):\nabla f\in L_{2}(\Omega)\rbrace, 
 \end{equation} endowed with the norm $$\parallel f\parallel ^{2}_{\mathcal{H}^{1}(\Omega)} := \parallel \omega_{2}^{-1}f \parallel ^{2}_{L^{2}(\Omega)} + \parallel\nabla f\parallel^{2}_{L^{2}(\Omega)}.$$
\end{itemize}

For the functions from $\mathcal{H}^{1}(\Omega),$ the semi-norm $$|f|_{\mathcal{H}^{1}(\Omega)}:=\Vert \nabla f\Vert_{L_{2}(\Omega)}$$ is equivalent\cite[ Chapter XI, Part B, \S 1]{lions} to the norm $\Vert\cdot\Vert_{\mathcal{H}^{1}(\Omega)}.$ 

The space $\mathcal{D}(\mathbb{R}^{2})$ is dense\cite{amrouche1} in $\mathcal{H}^{1}(\mathbb{R}^{2})$. This implies that the dual space of $\mathcal{H}^{1}({\mathbb{R}^{2}}),$ denoted by $\mathcal{H}^{-1}(\mathbb{R}^{2}),$ is a space of distributions. Note that $\mathcal{D}(\overline{\Omega})$ is dense in $\mathcal{H}^{1}(\Omega)$\cite[Section 1]{amrouche00}. 

Let us introduce $\widetilde{{\mathcal{H}}}^{1}(\Omega)$ as the completion of ${\mathcal{D}}(\Omega)$ in ${\mathcal{H}}^{1}(\mathbb{R}^{2})$; let $\widetilde{{\mathcal{H}}}^{-1}(\Omega) := [{\mathcal{H}}^{1}(\Omega)]^{*}$ and ${\mathcal{H}}^{-1}(\Omega) := [ \widetilde{{\mathcal{H}}}^{1}(\Omega)]^{*}$ be the corresponding dual spaces. 

The inclusion $L_{2}(\omega_{2};\Omega)\subset\mathcal{H}^{-1}(\Omega)$ holds and a distribution $f$ in the dual space $\widetilde{\mathcal{H}}^{-1}(\Omega)$ can be represented as $f=\sum_{i=1}^{2}\partial_{x_{i}}g_{i}+f_{0},$ where $g_{i}\in L_{2}(\mathbb{R}^{2})$ and is zero outside $\Omega,~f_{0}\in L_{2}(\omega_{2};\Omega),$ cf. e.g., \cite[Section 2.5]{nedelec3}. This implies that $\mathcal{D}(\Omega)$ is dense in $\widetilde{\mathcal{H}}^{-1}(\Omega)$ and $\mathcal{D}(\mathbb{R}^{2})$ is dense in $\mathcal{H}^{-1}(\mathbb{R}^{2}).$ 

For $u\in\mathcal{H}^{1}(\Omega)$ the operator $\mathcal{A}$ is well defined\cite{exterior} in the distributional sense as long as the variable coefficient $a\in L^{\infty}(\Omega),$ as
\begin{equation}\label{ch5exbilinearA}
\langle \mathcal{A}u, v \rangle_{\Omega} = - \langle a\nabla u, \nabla v\rangle_{\Omega} =-\mathcal{E}(u,v) \quad \forall v \in \mathcal{D}(\Omega),
\end{equation}
where 
\begin{equation}\label{ch5exE}
\mathcal{E}(u,v) : = \int_{\Omega} E(u,v) (x) dx; \quad\quad E(u,v)(x):= a(x) \nabla u(x)\nabla v(x). 
\end{equation}

The boundedness of the variable coefficient $a(x)$, reflected in \ref{cond1}, is required to guaranteed the continuity of the functional $\mathcal{E}(u,v):\mathcal{H}^{1}(\Omega)\times \widetilde{\mathcal{H}}^{1}(\Omega)\longrightarrow \mathbb{R}$ and, thus, the continuity of continuity of the operator $\mathcal{A}:\mathcal{H}^{1}(\Omega)\longrightarrow\mathcal{H}^{-1}(\Omega)$ which gives the distributional form of the operator $\mathcal{A}$ given in \eqref{ch4operatorA}.

The Trace Theorem can be extended to weighted Sobolev spaces, i.e. if $w\in \mathcal{H}^{1}(\Omega)$ then $\gamma^{\pm}w\in H^{\frac{1}{2}}(S)$\cite[Section 2.2.2]{MKW20tr} and the trace operators $\gamma^{\pm}$ are linear continuous and surjective. 

The conormal derivative operator acting on $S$, understood in the trace sense, is given by
\begin{equation}\label{ch5conormal}
T^{\pm}[u(x)] :=\sum_{i=1}^{2}a(x)n_{i}(x)\gamma^{\pm}\left( \dfrac{\partial u}{\partial x_{i}}\right)=a(x)\gamma^{\pm}\left( \dfrac{\partial u(x)}{\partial n(x)}\right),
\end{equation}
where $n(x)$ is the exterior unit normal vector to the domain $\Omega$ at a point $x\in S$. It is well known that for $u\in \mathcal{H}^{1}(\Omega)$, the classical co-normal derivative operator may not exist\cite{traces}. However, one can overcome this difficulty by introducing the following function space\cite{exterior, giroire2, hanouzet} for the operator $\mathcal{A}$,
\begin{equation}\label{ch5spaceHA}
\mathcal{H}^{1,0}(\Omega;\mathcal{A}) := \lbrace g\in \mathcal{H}^{1}(\Omega): \mathcal{A}g\in L^{2}(\omega; \Omega)\rbrace
\end{equation}
endowed with the norm
\begin{center}
$\parallel g\parallel ^{2}_{\mathcal{H}^{1,0}(\Omega;\mathcal{A})} := \parallel g \parallel ^{2}_{\mathcal{H}^{1}(\Omega)} + \parallel \omega_{2} \mathcal{A}g\parallel^{2}_{L^{2}(\Omega)}.$
\end{center}
When $u\in \mathcal{H}^{1,0}(\Omega; \mathcal{A})$, we can correctly define the conormal derivative $T^{+}u\in H^{-\frac{1}{2}}(S)$ using the first Green identity as follows\cite[Section 2]{exterior},
\begin{equation}\label{ch5green1}
\langle T^{+}u, w\rangle_{S}:= \pm \int_{\Omega^{\pm}}[(\gamma^{+}_{-1}\omega)\mathcal{A}u +E(u,\gamma_{-1}^{+}w)]\,\, dx;\,\, \text{for all}\,\, w\in H^{\frac{1}{2}}(S),
\end{equation}
where $\gamma_{-1}^{+}: H^{\frac{1}{2}}(S)\rightarrow \mathcal{H}^{1}(\Omega)$ is a continuous right inverse to the trace operator $\gamma^{+}:\mathcal{H}^{1}(\Omega) \longrightarrow H^{\frac{1}{2}}(S)$.

The operator $T^{+}:\mathcal{H}^{1,0}(\Omega; \mathcal{A})\longrightarrow H^{\frac{-1}{2}}(S)$ is linear, bounded and gives a continuous extension on $\mathcal{H}^{1,0}(\Omega; \mathcal{A})$ of the classical co-normal derivative operator \eqref{ch5conormal}. We remark that when $a\equiv 1$, the operator $T^{+}$ becomes $T^{+}_{\Delta}:= n\cdot \nabla$, which is the continuous extension on $\mathcal{H}^{1,0}(\Omega; \Delta)$ of the classical normal derivative operator. Furthermore, the first Green identity holds\cite[Section 2]{exterior} for any  distribution $u\in\mathcal{H}^{1,0}(\Omega;\mathcal{A}),$ 
\begin{equation}\label{ch5GF1}
\langle T^{+}u ,\gamma^{+}v\rangle_{S}=\displaystyle\int_{\Omega}[v\mathcal{A}u + E(u,v)] dx, \quad \forall v\in\mathcal{H}^{1}(\Omega).
\end{equation}
As a consequence of the first Green identity \eqref{ch5GF1} and the symmetry of $E(u,v)$, the second Green identity holds for any $u,v\in\mathcal{H}^{1,0}(\Omega;\mathcal{A})$ 
\begin{equation}\label{ch5secondgreen}
\displaystyle\int_{\Omega}\left[v\mathcal{A}u - u\mathcal{A}v\right]\,dx= \int_{S}\left[\gamma^{+}v\,T^{+}u-\gamma^{+}u\,T^{+}v\right]\,dS(x).
\end{equation}

\section{Boundary Value Problem}
We aim to derive boundary-domain integral equation systems for the following \textit{Dirichlet} boundary value problem defined in an exterior $\Omega$. Given $f\in L^{2}(\omega_{2};\Omega)$ and $\varphi_{0}\in H^{\frac{1}{2}}(S)$, we seek a function $u\in\mathcal{H}^{1,0}(\Omega;\mathcal{A})$ such that 
\begin{subequations}\label{ch4BVP}
\begin{align}
\mathcal{A}u&=f,\hspace{1em}\text{in}\hspace{1em}\Omega\label{ch4BVP1};\\
\gamma^{+}u &= \varphi_{0},\hspace{1em}\text{on}\hspace{1em} S\label{ch4BVP2}
\end{align}
where equation \eqref{ch4BVP1} is understood in the weak sense, the Dirichlet condition \eqref{ch4BVP2} is understood in the trace sense.
\end{subequations}

Let us denote the left hand side operator of the Dirichlet problem as 
\begin{equation}\label{oper2}
\mathcal{A}_{D}:[\mathcal{A},\gamma^{+}]:\mathcal{H}^{1,0}(\Omega;\mathcal{A})\longrightarrow L_{2}(\omega_{2};\Omega)\times H^{\frac{1}{2}}(S),
\end{equation}
 By using variational settings and the Lax-Milgram Lemma, similar to the proof in \cite[Theorem A.1]{exterior} for the three dimensional case it is possible to prove that the operator \eqref{oper2} is continuously invertible and thus the unique solvability of the BVP \eqref{ch4BVP1}-\eqref{ch4BVP2} follows.

We define a parametrix (Levi function) $P(x,y)$ for a differential operator $\mathcal{A}_{x}$ differentiating with respect to $x$ as a function on two variables that satisfies
 \begin{equation}\label{parametrixdef}
 \mathcal{A}_{x}P(x,y) = \delta(x-y)+R(x,y).
 \end{equation}
 where $\delta (.)$ is a Dirac-delta distribution, while $R(x,y)$ is a remainder possessing at most a weak (integrable) singularity at $ x=y $.
 
In this paper we will use the same parametrix as in \cite{carloszenebe, carloscomp}
\begin{equation*}\label{ch5P2002}
P(x,y)=\dfrac{1}{a(x)} P_\Delta(x-y),\hspace{1em}x,y \in \mathbb{R}^{2},
\end{equation*} 
whose corresponding remainder is 
\begin{equation}
\label{ch53.4} R(x,y)
=-\sum\limits_{i=1}^{2}\dfrac{\partial}{\partial x_{i}}\left(  \frac{1}{a(x)} \frac{\partial a(x)}{\partial x_i} P_\Delta(x-y)\right),\quad \text{where}\,\,\, P_{\Delta}(x-y)=\dfrac{1}{2\pi a(y)}\mathrm{log}|x-y|\
\,\,\;\;\;x,y\in {\mathbb R}^2.
\end{equation}
Here, $P_{\Delta}(x-y)$ represents the fundamental solution for the Laplace equation in two dimensions. 

\section{Volume and surface potentials}
The parametrix-based logarithmic and remainder potential operators are respectively defined, similar to \cite{{mikhailov1},{carlosext}} in the 3D case for $y\in\mathbb R^2$, as
\begin{align*}
\mathcal{P}\rho(y)&:=\displaystyle\int_{\Omega} P(x,y)\rho(x)\hspace{0.25em}dx\\
\mathcal{R}\rho(y)&:=\displaystyle\int_{\Omega} R(x,y)\rho(x)\hspace{0.25em}dx.
\end{align*}

Note that when, in the definition above, $\Omega=\mathbb{R}^{2}$ we will denote the operators $\mathcal{P}$ and $\mathcal{R}$ by $\textbf{P}$ and $\textbf{R}$ respectively, and relations to \eqref{ch4relP} and \eqref{ch4relR} hold for them as well.

The parametrix-based single layer and double layer  surface potentials are defined for $y\in\mathbb R^2:y\notin S $, as 
\begin{equation*}\label{ch4SL}
V\rho(y):=-\int_{S} P(x,y)\rho(x)\hspace{0.25em}dS(x),
\end{equation*}
\begin{equation*}\label{ch4DL}
W\rho(y):=-\int_{S} T_{x}^{+}P(x,y)\rho(x)\hspace{0.25em}dS(x).
\end{equation*}

We also define the following pseudo-differential operators associated with direct values of the single and  double layer potentials and with their conormal derivatives, for $y\in S$,
\begin{align}
\mathcal{V}\rho(y)&:=-\int_{S} P(x,y)\rho(x)\hspace{0.25em}dS(x),\nonumber \\
\mathcal{W}\rho(y)&:=-\int_{S} T_{x}P(x,y)\rho(x)\hspace{0.25em}dS(x),\nonumber\\
\mathcal{W'}\rho(y)&:=-\int_{S} T_{y}P(x,y)\rho(x)\hspace{0.25em}dS(x),\nonumber\\
\mathcal{L}^{\pm}\rho(y)&:=T_{y}^{\pm}{W}\rho(y)\nonumber.
\end{align}

The operators $\mathcal P, \mathcal R, V, W, \mathcal{V}, \mathcal{W}, \mathcal{W'}$ and $\mathcal{L}$ can be expressed in terms the volume and surface potentials and operators associated with the Laplace operator, as follows
\begin{align}
\mathcal{P}\rho&=\mathcal{P}_{\Delta}\left(\dfrac{\rho}{a}\right),\label{ch4relP}\\
\mathcal{R}\rho&=\nabla\cdot\left[\mathcal{P}_{\Delta}(\rho\,\nabla \ln a)\right]-\mathcal{P}_{\Delta}(\rho\,\Delta \ln a),\label{ch4relR}\\
V\rho &= V_{\Delta}\left(\dfrac{\rho}{a}\right),\label{ch4relSL}\\
\mathcal{V}\rho &= \mathcal{V}_{\Delta} \left( \dfrac{\rho}{a}\right),\label{ch4relDVSL}\\
W\rho &= W_{\Delta}\rho -V_{\Delta}\left(\rho\frac{\partial \ln a}{\partial n}\right),\label{ch4relDL}\\
\mathcal{W}\rho &= \mathcal{W}_{\Delta}\rho -\mathcal{V}_{\Delta}\left(\rho\frac{\partial \ln a}{\partial n}\right),\label{ch4relDVDL}\\
\mathcal{W}'\rho &= a \mathcal{W'}_{\Delta}\left(\dfrac{\rho}{a}\right),\label{ch4relTSL} \\
\mathcal{L}^{\pm}\rho &= \widehat{\mathcal{L}}\rho - aT^{\pm}_\Delta V_{\Delta}\left(\rho\frac{\partial \ln a}{\partial n}\right),
\label{ch4relTDL}\\
\widehat{\mathcal{L}}\rho &:= a\mathcal{L}_{\Delta}\rho.\label{ch4hatL}
\end{align}

The symbols with the subscript $\Delta$ denote the analogous operator for the constant coefficient case, $a\equiv 1$. Furthermore, by the Liapunov-Tauber theorem \cite{tauber2}, $\mathcal{L}_{\Delta}^{+}\rho = \mathcal{L}_{\Delta}^{-}\rho = \mathcal{L}_{\Delta}\rho$.
 
 To guarantee the continuity, and thus boundedness, of the surface and volume integral operators, we will need to impose conditions on the variable coefficient as well as on its derivatives. 

 \begin{cond}\label{cond2} To obtain boundary-domain integral equations, we will assume the following condition further on unless stated otherwise:
\begin{equation*}
a\in \mathcal{C}^{1}(\mathbb{R}^{2})\quad and \quad \omega_{2}\nabla a \in L^{\infty}(\mathbb{R}^{2}).
\end{equation*}
\end{cond}

 \begin{remark}\label{ch5remmult}
If $a$ satisfies \eqref{cond1} and \eqref{cond2}, then $\parallel ga\parallel_{\mathcal{H}^{1}(\Omega)} \leq k_{1}\parallel g \parallel_{\mathcal{H}^{1}(\Omega)} $, $\parallel g/a \parallel_{\mathcal{H}^{1}(\Omega)}\leq k_{2}\parallel g \parallel_{\mathcal{H}^{1}(\Omega)}$ where the constants $k_{1}$ and $k_{2}$ do not depend on $g\in\mathcal{H}^{1}(\Omega)$, i.e., the functions $a$ and $1/a$ are multipliers in the space $\mathcal{H}^{1}(\Omega)$. Furthermore, as long as $a\in \mathcal{C}^{1}(S)$, then $\dfrac{\partial a}{\partial n}$ is also a multiplier. 
\end{remark}

 \begin{theorem}\label{thmVW} The following operators are continuous under Condition \ref{cond2},
 \begin{align*}
 V&:H^{-\frac{1}{2}}_{*}(S) \longrightarrow \mathcal{H}^{1}(\Omega),\hspace{0.5em},\\
 W&:H^{\frac{1}{2}}(S) \longrightarrow \mathcal{H}^{1}(\Omega).\hspace{0.5em}
 \end{align*}
 \end{theorem}
 
 \begin{proof}
Let us first prove the mapping property for the operator $V$. Let $g \in H_{*}^{-\frac{1}{2}}(S)\subset  H^{-\frac{1}{2}}(S)$, then $\dfrac{g}{a}$ also belongs to $H^{-1/2}(S)$ by virtue of Remark \ref{ch5remmult} and Condition \ref{cond1}. Then, relation \eqref{ch4relSL} along with the mapping property $V_{\Delta}:H^{-1/2}(S)\longrightarrow H^{1}(\Omega)\subset\mathcal{H}^{1}(\Omega;\Delta)$\cite[Lemma 6.6]{steinbach} imply that $Vg = V_{\Delta}\left(g/a\right)\in \mathcal{H}^{1}(\Omega;\Delta)$ from where it follows the result $Vg\in \mathcal{H}^{1}(\Omega)$. 

Let us prove now the result for the operator $W$. If $g \in H^{1/2}(S)$, then $\partial_{n}(\ln a)g$ also belongs to $H^{1/2}(S)$ in virtue of Remark \ref{ch5remmult} and Condition \ref{cond1}. Then, relation \eqref{ch4relDL} along with the mapping properties $V_{\Delta}:H^{-1/2}(S)\longrightarrow \mathcal{H}^{1}(\Omega;\Delta)$\cite[Lemma 6.6]{steinbach} and $W_{\Delta}:H^{1/2}(S)\longrightarrow H^{1}(\Omega)$\cite[Lemma 6.10]{steinbach} imply that $Wg\in H^{1}(\Omega)$ from where it follows that $Wg\in H^{1}(\Omega)\subset \mathcal{H}^{1}(\Omega)$.
 \end{proof}
 
 \begin{corollary}\label{ch5thmapVW} The following operators are continuous under the Condition \ref{cond2},
 \begin{align}
 V&:H^{-\frac{1}{2}}_{*}(S) \longrightarrow \mathcal{H}^{1,0}(\Omega;\mathcal{A}),\hspace{0.5em},\label{ch5opVcont}\\
 W&:H^{\frac{1}{2}}(S) \longrightarrow \mathcal{H}^{1,0}(\Omega;\mathcal{A}),\hspace{0.5em}\label{ch5opWcont}
 \end{align}
 \end{corollary}   
 
\begin{proof}
Let us prove first the mapping property \eqref{ch5opVcont}. Let $g\in H_{*}^{-1/2}(S)$. From  Theorem \ref{thmVW}, $Vg\in \mathcal{H}^{1}(\Omega)$. Hence, it suffices to prove that $\mathcal{A}Vg\in L^{2}(\omega;\Omega)$. 

Differentiating using the product rule for some smooth function $h$, we can write
\begin{equation}\label{ch5cordA}
\mathcal{A}h = \nabla a \nabla h + a \Delta h.
\end{equation}

Taking into account relation \eqref{ch4relSL} and applying \eqref{ch5cordA} to $h=V_{\Delta}(g/a)$, we get 
\begin{equation}
\mathcal{A}V_{\Delta}\left(\dfrac{g}{a}\right) = \sum_{i=1}^{3}\dfrac{\partial a}{\partial y_{i}} \dfrac{\partial V_{\Delta}}{\partial y_{i}} \left(\dfrac{g}{a}\right) + a \Delta V_{\Delta}\left(\dfrac{g}{a}\right) = \sum_{i=1}^{3}\dfrac{\partial a}{\partial y_{i}} \dfrac{\partial V_{\Delta}}{\partial y_{i}} \left(\dfrac{g}{a}\right)=\nabla a\nabla V(g).
\end{equation}
By virtue of the mapping property for the operator $V$ provided by Theorem \ref{thmVW}, the last term belongs to $L^{2}(\omega;\Omega)$ due to the fact that $Vg\in \mathcal{H}^{1}(\Omega)$, and thus the components of $\nabla V(g)$ belong to $L^{2}(\omega;\Omega)$. The term  $\nabla a$ acts as a multiplier in the space $L^{2}(\omega;\Omega)$ due to Condition \ref{cond1}. On the other hand, the term $a\Delta V_{\Delta}(g/a)$ vanishes on $\Omega$ since $V_{\Delta}(\cdot)$ is the single layer potential for the Laplace equation, i.e., $V_{\Delta}(g/a)$ is a harmonic function. This completes the proof for the operator $V$. 

The proof for the operator $W$ follows from a similar argument. 
\end{proof}
\begin{lemma}\label{LL2}
Let $g\in L^{2}(\omega^{-1}; \mathbb{R}^{2})$ and let Condition \ref{cond2} hold. Then, the components of $g \cdot \nabla (\ln a)$ belong to $L^{2}(\mathbb{R}^{2})$.
\end{lemma}

\begin{proof}
If $g\in L^{2}(\omega^{-1}; \mathbb{R}^{2})$, then $\parallel \omega_{2}^{-1}g \parallel_{L^{2}(\mathbb{R}^{2})}<\infty$. Let $C_{3}:=\parallel g \omega_{2}^{-1}\parallel_{L^{2}(\Omega)}$. On the other hand, the components of $g \cdot \nabla (\ln a)$ can be written as $\dfrac{g}{a} \partial_{i}a$. Since Condition \ref{cond2} holds, $(\omega_{2}\partial_{i}a)\in L^{\infty}(\mathbb{R}^{2})$, we can define a new constant $C_{4} = \max_{i=1,2} \parallel \omega_{2}\partial_{i}a \parallel_{L^{\infty}(\mathbb{R}^{2})}$.
Taking this into account, let us work out the norm in $L^{2}(\mathbb{R}^{2})$ of $\dfrac{g}{a} \partial_{i}a$.
\begin{align*}
\parallel \dfrac{g}{a} \partial_{i}a \parallel_{L^{2}(\mathbb{R}^{2})} &= \parallel (\omega_{2}^{-1}g) \dfrac{1}{a} (\omega_{2}\partial_{i}a)\parallel_{L^{2}(\mathbb{R}^{2})} \leq \dfrac{C_{4}}{C_{1}}\parallel \omega_{2}^{-1}g \parallel_{L^{2}(\mathbb{R}^{2})}\leq \infty,
\end{align*}
from where it follows the result.  
\end{proof}
\begin{theorem}\label{ch4thmappingPR} The following operators are continuous under Condition \ref{cond2},
\begin{align}
\textbf{P}&: \mathcal{H}_{*}^{-1}(\mathbb{R}^{2})\longrightarrow\mathcal{H}^{1}(\mathbb{R}^{2}),\label{ch5oppcontbf}\\
\mathcal{P}&: \widetilde{\mathcal{H}}^{-1}_{*}(\Omega)\longrightarrow\mathcal{H}^{1}(\mathbb{R}^{2})\label{ch5opPcont},\\
\textbf{R}&: L^{2}(\omega_{2};\mathbb{R}^{2})\longrightarrow\mathcal{H}^{1}(\mathbb{R}^{2}).\label{ch5opRcontbf}
\end{align}
\end{theorem}

\begin{proof} 
 Let $g\in \mathcal{H}_{*}^{-1}(\mathbb{R}^{2})\subset \mathcal{H}^{-1}(\mathbb{R}^{2})$. Then, by virtue of the relation \eqref{ch4relP},
$\textbf{P}g = \textbf{P}_{\Delta}(g/a)$ and clearly $(g/a)\in \mathcal{H}^{-1}(\mathbb{R}^{2})$. Therefore, the continuity of the operator $\textbf{P}$ follows from  the continuity of $\textbf{P}_{\Delta}:\mathcal{H}^{-1}(\mathbb{R}^{2})\longrightarrow \mathcal{H}^{1}(\mathbb{R}^{2})$ \cite[Theorem III.2]{hanouzet}, which at the same time implies the continuity of the operator \eqref{ch5opPcont}.

Let us prove now the continuity of the operator $\textbf{R}$. First, the relation \eqref{ch4relR} can be used to express the operator  gives $\mathbf{R}$ in terms of the operator $\mathbf{P}_{\Delta}$ for some $g\in L^{2}(\omega^{-1}; \mathbb{R}^{2})$, as follows
\begin{align}
\mathbf{R}g(y)&=-\nabla\cdot \mathbf{P}_{\Delta}(g\cdot \nabla (\ln a))(y)\nonumber=-\sum_{i=1}^{2}\dfrac{\partial}{\partial y_{i}}\mathbf{P}_{\Delta}\left(g\cdot \dfrac{\partial(\ln a)}{\partial x_{i}}\right)(y)\nonumber\\
&=-\sum_{i=1}^{2}\mathbf{P}_{\Delta}\left[\dfrac{\partial}{\partial x_{i}}\left(g\cdot \dfrac{\partial(\ln a)}{\partial x_{i}}\right)\right](y):=-\mathbf{P}_{\Delta}g^{*}(y).
\end{align}

From Lemma \ref{LL2}, we know $(g\cdot \nabla (\ln a))\in L^{2}(\mathbb{R}^{2})$. Consequently, $g^{*} : = \nabla \cdot (g\cdot \nabla (\ln a))\in H^{-1}(\mathbb{R}^{2})$. Since ${\mathbf{P}_{\Delta}:H^{-1}(\mathbb{R}^{2})\longrightarrow H^{1}(\mathbb{R}^{2})\subset \mathcal{H}^{1}(\mathbb{R}^{2})}$ is continuous\cite[Theorem III.2]{hanouzet} or \cite[Theorem 6.1]{steinbach}, the operator ${\mathbf{R}: L^{2}(\omega^{-1}; \mathbb{R}^{2})\longrightarrow \mathcal{H}^{1}(\mathbb{R}^{2})}$ is also continuous. \end{proof}

When applying similar arguments as in the previous proofs, partial derivatives of second order of the coefficient $a(x)$ appear, it will be necessary to impose additional boundedness conditions on $a(x)$. 
  \begin{cond}\label{cond3}
 In addition to Condition \ref{cond1} and Condition \ref{cond2}, we will also sometimes assume the following 
 \begin{equation*}
 (\omega_{2})^{2}\Delta a\in L^{\infty}(\mathbb{R}^{2}).
 \end{equation*}
 \end{cond}
 
\begin{theorem}\label{thPRH10} The following operators are continuous under Condition \ref{cond2} and Condition \ref{cond3},
\begin{align}
\mathcal{P}&: L^{2}(\omega_{2};\Omega)\longrightarrow\mathcal{H}^{1,0}(\mathbb{R}^{2};\mathcal{A}),\label{ch5opPH10}\\
\mathcal{R}&: \mathcal{H}^{1}(\Omega) \longrightarrow \mathcal{H}^{1,0}(\Omega;\mathcal{A})\label{ch5opRH10}.
\end{align}
\end{theorem}

\begin{proof}
To prove the continuity of the operator \eqref{ch5opPH10}, we consider a function $g\in L^{2}(\omega_{2}; \Omega)$ and its extension by zero to $\mathbb{R}^{2}$ which we denote by $\widetilde{g}$. Clearly, $\widetilde{g}\in L^{2}(\omega_{2}; \mathbb{R}^{2})\subset \mathcal{H}^{-1}(\mathbb{R}^{2})$. Taking into account the relation \eqref{ch4relP} and Theorem \ref{ch4thmappingPR}, we obtain that $\mathcal{P}_{\Delta}(g/a) = \mathbf{P}_{\Delta}(\widetilde{g}/a)\in \mathcal{H}^{1}(\mathbb{R}^{2})$.  Hence, it remains to prove that $\mathcal{A}\mathbf{P}_{\Delta}(\widetilde{g}/a)\in L^{2}(\omega_{2}; \mathbb{R}^{2})$. 
\begin{equation}\label{id}
\mathcal{A}\mathbf{P}\widetilde{g} = \mathcal{A}\mathbf{P}_{\Delta}(\widetilde{g}/a) = \widetilde{g}+\nabla a\cdot \nabla\mathbf{P}_{\Delta}(\widetilde{g}/a). 
\end{equation}
Since Condition \ref{cond2} is satisfied, the multiplication by $\nabla a$ in the second term of \eqref{id} behaves as a multiplier in the space $L^{2}(\omega_{2}; \mathbb{R}^{2})$. Therefore, we conclude that $\mathcal{A}\mathcal{P}g(y)\in L^{2}(\omega_{2},\Omega)$ and therefore $\mathcal{P}g\in \mathcal{H}^{1,0}(\Omega, \mathcal{A})$.

Finally, let us prove the continuity of the operator \eqref{ch5opRH10}. The continuity of the operator $\mathcal{R}:\mathcal{H}^{1}(\Omega)\longrightarrow \mathcal{H}^{1}(\Omega)$ follows from the continuous embedding ${\mathcal{H}^{1}(\Omega)\subset L^{2}(\omega^{-1};\Omega)}$ and the continuity of the operator \eqref{ch5opRcontbf}. Hence, we only need to prove that $\mathcal{A}\mathcal{R}g\in L^{2}(\omega;\Omega)$. For $g\in \mathcal{H}^{1}(\Omega)$ we have
\begin{align*}
\mathcal{A}\mathcal{R}g =\nabla a \cdot \nabla \mathcal{R}g+ a\Delta \mathcal{R}g. 
\end{align*} 
As $\mathcal{R}g\in \mathcal{H}^{1}(\Omega)$, we only need to prove that $\Delta\mathcal{R}g(y)\in L^{2}(\omega_{2}; \Omega)$. Using the relation \eqref{ch4relR}, we obtain that
\begin{align*}
\Delta \mathcal{R}g(y) &=\Delta \left[ -\nabla\cdot \mathcal{P}_{\Delta} (g \nabla (\ln a))\right]=  -\nabla\cdot \Delta\mathcal{P}_{\Delta} (g \nabla (\ln a))= -\nabla\cdot(g \nabla (\ln a)), 
\end{align*}
since $g\in\mathcal{H}^{1}(\Omega)$, then $g\in L^{2}(\omega_{2}, \Omega)$. $\nabla (\ln a)$ is a multiplier in the space $\mathcal{H}^{1}(\Omega)$ by virtue of the Condition \ref{cond2}. Then $(g \nabla \ln a)\in \mathcal{H}^{1}(\Omega)$. Consequently, ${-\nabla\cdot(g \nabla \ln a)\in L^{2}(\omega_{2}; \Omega)}$ by virtue of Condition \ref{cond3}, from where it follows the result. 
\end{proof}

The following Corollary follows from the jump relations\cite[Lemma 6.7 and Lemma 6.11]{steinbach} of the harmonic potentials along with relations \eqref{ch4relSL} and \eqref{ch4relDL}. 
\begin{corollary}\label{ch4thjumps} Let  $\rho\in H^{-\frac{1}{2}}(S)$, $\tau\in H^{\frac{1}{2}}(S)$. Then the following operators jump relations hold $$\gamma^{\pm}V\rho=\mathcal{V}\rho,\quad\quad \gamma^{\pm}W\tau=\mp\dfrac{1}{2}\tau+\mathcal{W}\tau.$$
\end{corollary}

\section{BDIEs for the Dirichlet Problem} 
To derive a system of boundary-domain integral equations, we will need to obtain an integral representation formula for both, the solution $u$ and its trace $\gamma^{+u}$. We will use the potential operators introduced in the previous section to simplify the notation. 

First, let us apply the second Green identity \eqref{ch5secondgreen} with $v=P(\cdot,y)$ and any $u\in\mathcal{H}^{1,0}(\Omega;\mathcal{A})$. Keeping in mind the definition of parametrix \eqref{parametrixdef}, we obtain the third Green identity for the function $u\in\mathcal{H}^{1,0}(\Omega;\mathcal{A})$  
\begin{equation}\label{ch4green3}
u+\mathcal{R}u-VT^{+}u+W\gamma^{+}u=\mathcal{P}\mathcal{A}u,\hspace{1em}\text{in}\hspace{0.2em}\Omega.
\end{equation}
Applying the trace operator to the third Green identity \eqref{ch4green3}, and using the jump relations given in the Corollary \ref{ch4thjumps}, we obtain another representation formulae for the trace of the solution of the orignal BVP
\begin{equation}\label{ch43GG}
\dfrac{1}{2}\gamma^{+}u+\gamma^{+}\mathcal{R}u-\mathcal{V}T^{+}u+\mathcal{W}\gamma^{+}u=\gamma^{+}\mathcal{P}f,\hspace{0.5em}on\hspace{0.2em}S.
\end{equation}

To obtain a system of boundary-domain integral equation systems, we employ identity \eqref{ch4green3} in the domain $\Omega$, and identity \eqref{ch43GG} on $S$, substituting there the Dirichlet condition \eqref{ch4BVP2} and $T^{+}u = \psi$. We will consider the unknown function $\psi$ as formally independent of $u$ in $\Omega$. The BDIE system, so-called (M) reads
\begin{subequations}
\begin{align}
u+\mathcal{R}u-V\psi&=F_{0}\hspace{2em}\text{in}\hspace{0.5em}\Omega,\label{ch4SM12v}\\
\gamma^{+}\mathcal{R}u-\mathcal{V}\psi&=\gamma^{+}F_{0}-\varphi_{0}\label{ch4SM12g}\hspace{2em}\text{on}\hspace{0.5em}S,
\end{align}
\end{subequations}
where
\begin{equation}\label{ch4F0term}
F_{0}=\mathcal{P}f-W\varphi_{0}.
\end{equation}
We remark that $F_{0}$ belongs to the space $\mathcal{H}^{1,0}(\Omega;\mathcal{A})$ in virtue of the mapping properties of the surface and volume potentials, see Theorem \ref{thmVW} and Theorem \ref{ch4thmappingPR}. Furthermore, the Trace Theorem implies $\gamma^{+}F_{0}\in H^{\frac{1}{2}}(S).$ 

The system (M), given by \eqref{ch4SM12v}-\eqref{ch4SM12g} can be written in matrix notation as 
\begin{equation}\label{system}
\mathcal{M}\mathcal{U}=\mathcal{F},
\end{equation}
where $\mathcal{U}$ represents the vector containing the unknowns of the system,
\begin{equation*}
\mathcal{U}=(u,\psi)\in\mathcal{H}^{1,0}(\Omega;\mathcal{A})\times H^{-\frac{1}{2}}(S),
\end{equation*}
the right hand side vector is \[\mathcal{F}:= [ F_{0}, \gamma^{+}F_{0} - \varphi_{0} ]^{\top}\in\mathcal{H}^{1}(\Omega)\times H^{\frac{1}{2}}(S),\]
and the matrix operator $\mathcal{M}$ is defined by:
\begin{equation*}
   \mathcal{M}=
  \left[ {\begin{array}{ccc}
   I+\mathcal{R} & -V  \\
   \gamma^{+}\mathcal{R} & -\mathcal{V} 
  \end{array} } \right].
\end{equation*}

We note that the mapping properties of the operators involved in the matrix imply the continuity of the operator $\mathcal{M}$ under Condition \ref{cond2} and  Condition \ref{cond3}. 

Now that we have finally derived the system of BDIEs (M), let us prove that BVP\eqref{ch4BVP} in $\Omega$ is equivalent to the system of BDIEs \eqref{ch4SM12v}-\eqref{ch4SM12g}.

\section{Equivalence Theorem}
To prove the equivalence between the BDIEs (M) \eqref{ch4SM12v}-\eqref{ch4SM12g} and the original BVP \eqref{ch4BVP}. We will first prove a series of preliminary results. 

Let us consider a general BDIE obtained from  \eqref{ch4green3}, where $\gamma^{+}u$ and $T^{+}u$ have been replaced by $\Phi$ and $\Psi$. This substitution will allow us to consider $\Phi$ and $\Psi$ as unknowns formally segregated from $u$. 
\begin{equation}\label{ch4G3ind}
u+\mathcal{R}u-V\Psi+W\Phi=\mathcal{P}f,\hspace{0.5em}{\rm in\ }\Omega.
\end{equation}

Let us show that if a function $u\in\mathcal{H}^{1}(\Omega)$ satisfies \eqref{ch4G3ind}, then solves the PDE \eqref{ch4BVP1}. 

\begin{lemma}\label{ch4lema1}Let $u\in\mathcal{H}^{1}(\Omega)$, $f\in L_{2}(\omega_{2};\Omega)$, $\Psi\in H^{-\frac{1}{2}}(S)$ and $\Phi\in H^{\frac{1}{2}}(S)$ satisfying the relation \eqref{ch4G3ind} and let Condition \ref{cond2} and Condition \ref{cond3} hold. Then $u\in\mathcal{H}^{1,0}(\Omega;\mathcal{A})$ solves the equation $\mathcal{A}u=f$ in $\Omega$, and the following identity is satisfied,
\begin{equation}\label{ch4lema1.0}
V(\Psi- T^{+}u) - W(\Phi- \gamma^{+}u) = 0\hspace{0.5em}\text{in}\hspace{0.5em}\Omega.
\end{equation}
\end{lemma}

\begin{proof}
To prove that $u\in \mathcal{H}^{1,0}(\Omega;\mathcal{A})$, taking into account that by hypothesis $u\in \mathcal{H}^{1}(\Omega)$, so there is only left to prove that $\mathcal{A}u \in L^{2}(\omega; \Omega)$. 
Firstly we write the operator $\mathcal{A}$ as follows:
\begin{align*}
&\mathcal{A}u(x)=\Delta(au)(x) - \sum_{i=1}^{3}\dfrac{\partial}{\partial x_{i}}\left(u\left(
\dfrac{\partial a(x)}{\partial x_{i}}\right)\right).
\end{align*}

It is easy to see that the second term belongs to $L^{2}(\omega; \Omega)$. Keeping in mind Remark \ref{ch5remmult} and the fact that $u\in \mathcal{H}^{1}(\Omega)$, then we can conclude that the term $u\nabla a\in \mathcal{H}^{1}(\Omega)$ since due to Condition \ref{cond2}, $\nabla a$ is a multiplier in the space $\mathcal{H}^{1}(\Omega)$ and therefore $\nabla(u\nabla a)\in L^{2}(\omega; \Omega)$.

Now, we only need to prove that $\Delta (au)\in L^{2}(\omega;\Omega)$. To prove this we look at the relation \eqref{ch4G3ind} and we put $u$ as the subject of the formula. Then, we use the potential relations \eqref{ch4relP},  \eqref{ch4relSL} and \eqref{ch4relDL} 
\begin{equation}\label{ch5lema1.1}
u=\mathcal{P}f-\mathcal{R}u+V\Psi-W\Phi=\mathcal{P}_{\Delta}\left(\dfrac{f}{a}\right)-\mathcal{R}u+V_{\Delta}\left(\dfrac{\Psi}{a}\right)-W_{\Delta}\Phi+V_{\Delta} \left(\dfrac{\partial(\ln(a))}{\partial n}\Phi\right)
\end{equation}

In virtue of the Theorem \ref{thPRH10}, $\mathcal{R}u\in L^{2}(\omega;\Omega)$. Moreover, the terms in previous expression depending on $V_{\Delta}$ or $W_{\Delta}$ are harmonic functions and $\mathcal{P}_{\Delta}$ is the newtonian potential for the Laplacian, i.e. $\Delta\mathcal{P}_{\Delta}\left(\dfrac{f}{a}\right)=\dfrac{f}{a}$. Consequently, applying the Laplacian operator in both sides of \eqref{ch5lema1.1}, we obtain:
\begin{equation}\label{ch5lema1.2}
\Delta u = \dfrac{f}{a}-\Delta\mathcal{R}u. 
\end{equation}

Thus, $\Delta u \in L^{2}(\omega; \Omega)$ from where it immediately follows that $\Delta (au)\in L^{2}(\omega; \Omega)$. Hence $u\in \mathcal{H}^{1,0}(\Omega;\mathcal{A})$. 
We proceed subtracting \eqref{ch4green3} from \eqref{ch4G3ind} to obtain 
\begin{equation}\label{ch4lema1.3}
W(\gamma^{+}u-\Phi)-V(T^{+}u-\Psi)=\mathcal{P}(\mathcal{A}u-f).
\end{equation}
Let us apply relations \eqref{ch4relP}, \eqref{ch4relSL} and \eqref{ch4relDL} to \eqref{ch4lema1.3}, and then, apply the Laplace operator to both sides. Hence, we obtain
\begin{equation}\label{ch4lema1.5}
\mathcal{A}u-f=0,
\end{equation}
i.e., $u$ solves \eqref{ch4BVP1}.
Finally, substituting \eqref{ch4lema1.5} into \eqref{ch4lema1.3}, we prove \eqref{ch4lema1.0}.
\end{proof}

\begin{remark}\label{ch4remark}$(F_{0}, \gamma^{+}F_{0}-\varphi_{0})=0$ if and only if $(f, \varphi_{0})=0$
\end{remark}

\begin{proof}
Indeed the latter equality evidently implies the former, i.e., if $(f, \varphi_{0})=0$ then $(F_{0}, \gamma^{+}F_{0}-\varphi_{0})=0$. Conversely, supposing that $(F_{0}, \gamma^{+}F_{0}-\varphi_{0})=0$, then taking into account equation \eqref{ch4F0term} and applying Lemma \ref{ch4lema1} with $F_{0}=0$ as $u$, we deduce that $f=0$ and $W\varphi_{0}=0$ in $\Omega$. Now, the second equality, $\gamma^{+}F_{0}-\varphi_{0}=0$, implies that $\varphi_{0}=0$ on $S.$ 
\end{proof}

The equivalence (or not equivalence) essentially depends on the invertibility of the single layer potential operator $\mathcal{V}$. Given $D$ an open subset of $\mathbb{R}^{3}$ or a subset of $\mathbb{R}^{2}$ whose diameter is less than one. It is well known that the single layer potential $\mathcal{V}_{\Delta}:H^{-1/2}(\partial D)\longrightarrow H^{1/2}(\partial D)$ is an invertible operator\cite[Theorem 6.23]{steinbach}. 
However, when $\Omega$ is an unbounded domain, the assumption of $\text{diam}(\Omega)<1$ cannot be assume as it would violate the unboundedness property of $\Omega$. Therefore, we are forced to introduce the following spaces\cite{dufera,steinbach}
\begin{eqnarray*}
\mathcal{L}_{2,*}(\Omega)&:=&\lbrace f\in L_{2}(\omega_{2};\Omega):\langle f,1\rangle_{\Omega}=0\rbrace\\
\mathcal{H}^{1,0}_{*}(\Omega;\mathcal{A})&:=&\lbrace g\in\mathcal{H}^{1}(\Omega): \mathcal{A}g\in\mathcal{L}_{2,*}(\Omega)\rbrace,\\
H^{-\frac{1}{2}}_{*}(S)&:=&\lbrace \psi\in H^{-\frac{1}{2}}(S):\langle\psi,1\rangle_{S}=0\rbrace.
\end{eqnarray*} 

\begin{lemma}\label{L2}
Let $\Psi\in H^{-\frac{1}{2}}_{*}(S)$. Then, the operator $\mathcal{V}: H^{-\frac{1}{2}}_{*}(S)\longrightarrow H^{\frac{1}{2}}(S)$ is invertible. 
\end{lemma}

\begin{proof}
Applying the potential relation for the single layer potential \eqref{ch4relSL} to $\mathcal{V}\Psi^{*}$, we obtain that
$$ \mathcal{V}\Psi^{*} = \mathcal{V}_{\Delta}\left(\dfrac{\Psi^{*}}{a}\right).$$
Since the operator $\mathcal{V}_{\Delta}:H^{-1/2}_{*}(\partial \Omega)\longrightarrow H^{1/2}(\partial \Omega)$ is an invertible operator\cite[Theorem 6.23]{steinbach}, we only need to show that $\dfrac{\Psi^{*}}{a}\in H^{-1/2}_{*}(\partial \Omega)$. Let us show that $\dfrac{\Psi}{a}\in H_{*}^{-1/2}(S)$, taking into account \eqref{cond1} 
\begin{equation}
0=\dfrac{1}{C_{2}}\int_{S} \Psi \, \, dS \,\leq \int_{S} \dfrac{\Psi}{a}\, dS\, \leq  \,\dfrac{1}{C_{1}}\int_{S} \Psi \, \, dS = 0,\,\,\Rightarrow \langle \dfrac{\Psi}{a}, 1\rangle_{S} = 0 \, \Rightarrow \dfrac{\Psi}{a}\in H_{*}^{-1/2}(S).
\end{equation}
From, where it follows the result. 
\end{proof}

\begin{theorem}\label{ch4EqTh}
Let $f\in \mathcal{L}_{2,*}(\Omega)$ and $\phi_{0}\in H^{1/2}(S)$. Let Condition \ref{cond2} and Condition \ref{cond3} hold. 
\begin{enumerate}
\item[i)] If some $u\in\mathcal{H}^{1}(\Omega)$ solves the BVP \eqref{ch4BVP}, then $(u, \psi)\in\mathcal{H}^{1,0}_{*}(\Omega;\mathcal{A})\times H^{-\frac{1}{2}}_{*}(S)$ where
\begin{equation}\label{ch4eqcond}
\psi=T^{+}u,
\end{equation}
solves the BDIE system (M). 
\item[ii)] If a couple $(u, \psi)\in\mathcal{H}^{1}(\Omega)\times H_{*}^{-\frac{1}{2}}(S)$ solves the BDIE system (M) then $u\in\mathcal{H}^{1,0}_{*}(\Omega;\mathcal{A})$ solves the BVP and $\psi$ satisfies \eqref{ch4eqcond}.
\item[iii)] The system (M) is uniquely solvable. 
\end{enumerate}
\end{theorem}

\begin{proof}
First, let us prove item $i)$. Let $u\in \mathcal{H}^{1}(\Omega)$ be a solution of the boundary value problem \eqref{ch4BVP} with $f\in \mathcal{L}_{2,*}(\Omega)$. Then, $\mathcal{A}u \in \mathcal{L}_{2,*}(\Omega)$ and thus $u\in \mathcal{H}_{*}^{1,0}(\Omega;\mathcal{A})$. Since, in particular, $u\in \mathcal{H}^{1,0}(\Omega;\mathcal{A})$, we can correctly define the conormal derivative $\psi:=T^{+}u\in H^{-1/2}(S)$. Furthermore, replacing $v=1$ in the first Green identity \eqref{ch5GF1}, we obtain that 
$$\langle \psi, 1 \rangle_{L^{2}(S)} = \langle \mathcal{A}u, 1 \rangle_{L^{2}(\Omega)}=\langle f, 1 \rangle_{L^{2}(\Omega)}=0,$$
due to $f\in \mathcal{L}_{2,*}(\Omega)$. Therefore, $\psi \in H_{*}^{-1/2}(S)$. 
Then, it immediately follows from the third Green identities \eqref{ch4green3}-\eqref{ch43GG} that the couple $(u, \psi)$ solves BDIE system (M).

Let us prove now item $ii)$. Let the couple $(u, \psi)\in\mathcal{H}^{1}(\Omega)\times H_{*}^{-\frac{1}{2}}(S)$ solve the BDIE system (M). Taking the trace of the equation \eqref{ch4SM12v} and substract it from the equation \eqref{ch4SM12g}, we obtain
\begin{equation}\label{ch4M12a1}
\gamma^{+}u=\varphi_{0}, \hspace{1em} \text{on}\hspace{0.5em}S.
\end{equation}
Thus, the Dirichlet boundary condition in \eqref{ch4BVP2} is satisfied. 

We proceed using the Lemma \ref{ch4lema1} in the first equation of the system (M), \eqref{ch4SM12v}, which implies that $u\in \mathcal{H}^{1,0}(\Omega;\mathcal{A})$ is a solution of the equation \eqref{ch4BVP1}. Hence, since $\mathcal{A}u = f\in \mathcal{L}_{2,*}(\Omega)$, then $u\in \mathcal{H}_{*}^{1,0}(\Omega;\mathcal{A})$.  Moreover, from Lemma \ref{ch4lema1}, we also deduce that the following equality is satisfied
\begin{equation*}\label{ch4M12a2}
V(\psi - T^{+}u) - W(\varphi_{0} -\gamma^{+}u) = 0 \text{ in } \Omega.
\end{equation*}
By virtue of \eqref{ch4M12a1}, the second term of the previous equation vanishes. Hence,
\begin{equation*}\label{ch4M12a3}
V(\psi - T^{+}u)= 0, \quad \text{ in } \Omega.
\end{equation*}
Applying the trace operator to \eqref{ch4M12a3}, it results
\begin{equation*}\label{ch4M12a32}
\mathcal{V}(\dfrac{\psi - T^{+}u}{a})= 0, \quad \text{on}\,\,\, S.
\end{equation*}
Since $u\in \mathcal{H}_{*}^{1,0}(\Omega;\mathcal{A})$, the first Green identity holds and the conormal derivative $T^{+}u\in H_{*}^{-1/2}(S)$ is well defined. Let $\Psi:=\psi -T^{+}u$, clearly $\Psi\in H_{*}^{-1/2}(S)$, since $\psi\in H_{*}^{-1/2}(S)$ and $H_{*}^{-1/2}(S)$ is a vector space. Then, the hypotheses of Lemma \ref{L2} are satisfied and we can apply the inverse operator $\mathcal{V}^{-1}$ at both sides of \eqref{ch4M12a32} to obtain
\begin{equation}\label{ch4M12a4}
\psi = T^{+}u,\quad\text{ on } S.
\end{equation}

Item $iii)$ immediately follows from the uniqueness of the solution of the Dirichlet boundary value problem\cite[Theorem A.1]{exterior}\cite[Theorem 2]{dufera}.
\end{proof}

\section{Fredholm properties and Invertibility}
In this section, we follow a similar approach as in \cite[Section 7.2]{exterior}, we are going to benefit from the compactness properties of the operator $\mathcal{R}$ to prove invertibility of the operator $\mathcal{M}$. For this, we will split the operator $\mathcal{R}$ into two operators, one whose norm can be made arbitrarily small and another one that is contact. Then, we shall simply make use of the Fredholm alternative to prove the invertibility of the matrix operator $\mathcal{M}$ that defines the system of BDIEs. However, we can only split the operator $\mathcal{R}$ if the PDE satisfies the additional condition
\begin{cond}\label{ch5conda(x)4}
\begin{equation}
\lim_{x \rightarrow \infty} \omega(x)\nabla a(x) = 0.
\end{equation}
\end{cond}

\begin{lemma}\label{decR} Let Condition \ref{cond2} and Condition \ref{ch5conda(x)4} hold. Then, for any $\epsilon>0$ the operator $\mathcal{R}$ can be represented as $\mathcal{R}=\mathcal{R}_{s} + \mathcal{R}_{c}$, where $\parallel \mathcal{R}_{s} \parallel_{\mathcal{H}^{1}(\Omega)}< \epsilon$, while $\mathcal{R}_{c}: \mathcal{H}^{1}(\Omega) \rightarrow \mathcal{H}^{1}(\Omega)$ is compact. 
\end{lemma}

\begin{proof}
Let $B(0,r)$ be the ball centered at $0$ with radius $r$ big enough such that $S\subset B_{r}$. Furthermore, let $\chi\in \mathcal{D}(\mathbb{R}^{2})$ be a cut-off function such that $\chi=1$ in $S\subset B_{r}$, $\chi = 0$ in $\mathbb{R}^{2}\smallsetminus B_{2r}$ and $0\leq \chi(x) \leq 1$ in $\mathbb{R}^{2}$. Let us define by $\mathcal{R}_{c}g:=\mathcal{R}(\chi g)$, $\mathcal{R}_{s}g:=\mathcal{R}((1-\chi)g)$ where $g\in \mathcal{H}^{1}(\Omega)$.

We will prove first that the norm of $\mathcal{R}_{s}$ can be made infinitely small

\begin{align*}
 &\parallel \mathcal{R}_{s}g \parallel_{\mathcal{H}^{1}(\Omega)}= \parallel \sum_{i=1}^{2}\mathcal{P}_{\Delta}\left[ \dfrac{\partial}{\partial x_{i}}\left(\sum_{i=1}^{2}\dfrac{\partial(\ln a)}{\partial x_{i}}(1-\chi)g\right)\right]\parallel_{\mathcal{H}^{1}(\Omega)}\leq k \parallel \mathcal{P}_{\Delta} \parallel_{\widetilde{\mathcal{H}}^{-1}(\Omega)},\\ 
&\text{with}\quad k:=\sum_{i=1}^{2}\parallel \dfrac{\partial}{\partial x_{i}}\left(\sum_{i=1}^{2}\dfrac{\partial(\ln a)}{\partial x_{i}}(1-\chi)g\right)\parallel_{\widetilde{\mathcal{H}}^{-1}(\Omega)}\,\, \leq \sum_{i=1}^{2} \parallel \dfrac{\partial(\ln a)}{\partial x_{i}}(1-\chi)g\parallel_{{L}^{2}(\Omega)}\\
 & \hspace{4em} \leq 2 \parallel g \parallel_{L^{2}(\omega_{2}^{-1};\Omega)}\parallel \omega_{2} \nabla a\parallel_{L^{\infty}(\mathbb{R}^{2}\smallsetminus B_{r})}\,\,\leq\,\, 2 \parallel g \parallel_{\mathcal{H}^{1}(\Omega)}\parallel \omega_{2} \nabla a\parallel_{L^{\infty}(\mathbb{R}^{2}\smallsetminus B_{r})}.
\end{align*} 
Consequently, we have the following estimate:
\begin{align*}
\parallel \mathcal{R}_{s}g \parallel_{\mathcal{H}^{1}(\Omega)} &\leq 2 \parallel g \parallel_{\mathcal{H}^{1}(\Omega)}\parallel \omega_{2} \nabla a\parallel_{L^{\infty}(\mathbb{R}^{2}\smallsetminus B_{r})}\parallel \mathcal{P}_{\Delta} \parallel_{\widetilde{\mathcal{H}}^{-1}(\Omega)}.
\end{align*}

Using the previous estimate is easy to see that when $\epsilon\rightarrow +\infty$ the norm $\parallel \mathcal{R}_{s}g \parallel_{\mathcal{H}^{1}(\Omega)}$ tends to $0$. Hence, the norm of the operator $\mathcal{R}_{s}$ can be made arbitrarily small. 

To prove the compactness of the operator $\mathcal{R}_{c}g:=\mathcal{R}(\chi g)$, we recall that ${supp(\chi)\subset \bar{B}(0,2r)}$. Then, one can express $\mathcal{R}_{c}g:= \mathcal{R}_{\Omega_{r}}([\chi g\vert_{\Omega_{r}}])$ where the operator $\mathcal{R}$ is defined now over $\Omega_{r}:=\Omega\cap B_{2r}$ which is a bounded domain. As the restriction operator $\vert_{\Omega_{r}}:\mathcal{H}^{1}(\Omega)\longrightarrow \mathcal{H}^{1}(\Omega_{r})$ is continuous, the operator $\mathcal{R}_{c}g: L^{2}(\Omega_{r})\longrightarrow \mathcal{H}^{1}(\Omega_{r})$ is also continuous. Due to the boundedness of $\Omega_{r}$, we have $\mathcal{H}^{1}(\Omega_{r})= H^{1}(\Omega_{r})$ and thus the compactness of $\mathcal{R}_{c}g$ follows from the Rellich Theorem applied to the embedding $H^{1}(\Omega_{r})\subset L^{2}(\Omega_{r})$.
\end{proof}

\begin{corollary}\label{Rinv} Let Condition \ref{cond2} and Condition \ref{ch5conda(x)4} hold. Then, the operator $I+\mathcal{R}:\mathcal{H}^{1}(\Omega) \rightarrow \mathcal{H}^{1}(\Omega)$ is Fredholm with zero index. 
\end{corollary}   

\begin{proof}
Using the previous Lemma, we have $\mathcal{R}=\mathcal{R}_{s} + \mathcal{R}_{c}$ so $\parallel \mathcal{R}_{s}\parallel < 1$ hence $I+\mathcal{R}_{s}$ is invertible. On the other hand $\mathcal{R}_{c}$ is compact and hence $I + \mathcal{R}_{s}$ a compact perturbation of the operator $I+\mathcal{R}$, from where it follows the result.
\end{proof}

\begin{theorem}
If Condition \ref{cond2}, Condition \ref{cond3} and Condition \ref{ch5conda(x)4} hold, then the operator 
\begin{equation}\label{inv1}
\mathcal{M}:\mathcal{H}^{1,0}(\Omega,\mathcal{A})\times H^{-1/2}_{*}(\Omega) \rightarrow \mathcal{H}^{1,0}(\Omega,\mathcal{A})\times H^{1/2}(\Omega),
\end{equation}
is continuous and continuously invertible. 
\end{theorem}
\begin{proof}
First of all, the operator \eqref{inv1} can be represented in matrix notation as follows  
\begin{equation*}
   \mathcal{M}=
  \left[ {\begin{array}{ccc}
   I+\mathcal{R} & -V  \\
   \gamma^{+}\mathcal{R} & -\mathcal{V} 
  \end{array} } \right].
\end{equation*}
Theorem \ref{ch5thmapVW} and Theorem \ref{thPRH10} guarantee the continuity of the operator \eqref{inv1}. Let \[
   \mathcal{M}_{0}=
  \left[ {\begin{array}{cc}
   I+\mathcal{R}_{s}& -V  \\
   0 & -\mathcal{V}\\
  \end{array} } \right].
\]
The operators placed in the diagonal of the matrix $I:\mathcal{H}^{1,0}(\Omega,\mathcal{A})\longrightarrow\mathcal{H}^{1,0}(\Omega,\mathcal{A})$ and $\mathcal{V}:H^{-1/2}_{*}(S)\longrightarrow H^{1/2}(S)$
are continuous and invertible\cite[Theorem 6.22]{steinbach}.  Therefore, the operator $\mathcal{M}_{0}$ is invertible. 

Let us proceed applying Lemma \ref{decR} and decompose the operator $\mathcal{R}$ as $\mathcal{R} = \mathcal{R}_{c}+\mathcal{R}_{s}$. Using this decomposition, we define the following two new matrix operators
\begin{equation*}
   \mathcal{M}_{s}=\left[ {\begin{array}{cc}
   \mathcal{R}_{s}& 0  \\
   \gamma^{+}\mathcal{R}_{s} &0\\
  \end{array} } \right], \hspace{2cm}
  \mathcal{M}_{c}:=\left[ {\begin{array}{cc}
   \mathcal{R}_{c}& 0  \\
   \gamma^{+}\mathcal{R}_{c} &0\\
  \end{array} } \right].
\end{equation*}
Applying Lemma \ref{decR}, we can choose $\epsilon$ small enough so that $\parallel\mathcal{M}_{s}\parallel<\epsilon$. Then, we can choose $\epsilon$ to satisfy the inequality
$$ \parallel \mathcal{M}_{s}\parallel <\dfrac{1}{\parallel\mathcal{M}_{0}^{-1} \parallel}$$
so that the operator $\mathcal{M}_{s}+\mathcal{M}_{0}$ is continuously invertible. Additionally, we remark that the operator $\mathcal{M}_{c}$ is compact. Therefore, the operator $\mathcal{M}$ is a Fredholm operator with index zero for being the sum of an invertible operator and a compact operator. By the equivalence theorem, Theorem \ref{ch4EqTh}, the homogeneous BDIE \eqref{system} has only solution and therefore, the operator $\mathcal{M}$ is invertible by the Fredholm Alternative theorem. 
\end{proof}

A direct consequence of the previous argument and Lemma \ref{ch4lema1} is the following corollary. 
\begin{corollary} If Condition \ref{cond2}, Condition \ref{cond3} and Condition \ref{ch5conda(x)4} hold, then the operator 
\begin{equation}\label{inv2}
\mathcal{M}:\mathcal{H}^{1}(\Omega)\times H^{-1/2}_{*}(\Omega) \rightarrow \mathcal{H}^{1}(\Omega)\times H^{1/2}(\Omega),
\end{equation}
\end{corollary}

\section{Conclusions}
In this paper, we have considered a new parametrix for the Dirichlet problem with variable coefficient in two-dimensional unbounded domain, where the right hand side function is from $L_{2}(\omega_{2},\Omega)$ and the Dirichlet data from the space $H^{\frac{1}{2}}(S).$  A BDIEs for the original BVP has been obtained. Equaivalence of the BDIE system to the original BVP was proved in the case when the right hand side of the PDE is from $L^{2}(\omega_{2};\Omega)$ and the Dirichlet data from the space $H^{\frac{1}{2}}(S).$ 

Further generalised results for Lipschitz domains and BVPs with non-smooth coefficientscan also be obtain by using the generalised canonical conormal derivative operator defined in \cite{mikhailovlipschitz}. Moreover, these results can be generalised to systems of PDEs such as the Stokes system\cite{carlosstokes}.

\end{document}